\documentclass[12pt]{amsart}
\usepackage{amsmath, amssymb, color}
\usepackage{amsthm}
\usepackage{color}
\usepackage{multirow,bigdelim}
\usepackage{graphicx}
\usepackage{array}
\usepackage{ifpdf}
\usepackage{epsfig}
\usepackage{amscd}
\usepackage{graphics}
\usepackage{nccmath}
\numberwithin{equation}{section}
\theoremstyle{definition}
\newtheorem{thm}{Theorem}[section]

\newtheorem{lem}[thm]{Lemma}

\newtheoremstyle{mydefinition}
  {}   
  {}   
  {}  
  {0pt}       
  {} 
  {.}         
  {5pt plus 1pt minus 1pt} 
  {}          

\newtheoremstyle{citing}
  {}
  {}
  {\itshape}
  {}
  {\bfseries}
  {.}
  {.5em}
  {\thmnote{#3}}

\theoremstyle{citing}

\makeatother

\newcommand{\Kmn}{K^{m,n}}

\newcommand{\ve}{\varepsilon}

\newcommand{\pKmn}{\pi_1(S^4 \setminus {\rm int}N(\Kmn))}
\newcommand{\SLC}{{\rm SL}_2(\mathbb{C})}
\newcommand{\SLZ}{{\rm SL}_2(\mathbb{Z}_3)}
\newcommand{\pSLZ}{{\rm SL}^{\ast}_2(\mathbb{Z}_3)}

\title{Representations of branched twist spins with non-trivial center of order 2}
\author{Mizuki Fukuda}
\address{Mathematics for Advanced Materials Open Innovation Laboratory,
AIST, c/o AIMR, Tohoku University,
2-1-1, Katahira, Aoba-ku, Sendai, Miyagi 980-8577, Japan}
\email{mizuki.fukuda.d2@tohoku.ac.jp}

\begin{document}
\maketitle
\vspace{-5mm}
\begin{abstract}
It is known that a presentation of the knot group of a branched twist spin is obtained from a Wirtinger presentation of the original 1-knot group by adding a generator corresponding to a regular orbit of the circle action and a certain relator.
 In particular, the additional generator is an element of the center of the knot group.
 In this paper, we focus on $\SLZ$-representations and dihedral group representations.
 For the former case,  we give a sufficient  condition for the existence of an $\SLZ$-representation for a branched twist spin. For the latter case, we determine the number of even-ordered dihedral group representations of branched twist spins.
\end{abstract}

\section{Introduction}
A 2-knot is a smoothly embedded 2-sphere in the 4-sphere. Two 2-knots are equivalent if there exists a smooth ambient isotopy of $S^4$ such that it sends one to the other. 
 
Studies of classifying concrete 2-knots began with spun knots made by Artin~\cite{A}.
He gives a presentation of the knot group of a spun knot and show that the knot group of the spun knot is isomorphic to that of the original 1-knot. 
After several decades, Zeeman introduces a twist spun knot and gives a presentation of its knot group~\cite{Z}.
Then Pao introduces a wider class called a branched twist spin~\cite{Pa} and Plotnick gives a presentation of its knot group~\cite{Pl2}.
It is known that branched twist spins are characterized by 1-knots and two coprime integers $m,n$.
 We denote branched twist spins as $K^{m,n}$.
The precise definition of branched twist spins is given in Section 2.
Note that $K^{m,1}$ is the $m$-twist spun knot of $K$ and $K^{0,1}$ is the spun knot of $K$.
In \cite{G}, Gordon shows that $K^{m+1,1}$ is obtained from $K^{m,1}$ by the Gluck twist once and, in~\cite{F3}, the author generalizes his result for $K^{m,n}$.

Although the knot group of a branched twist spin has information of the original 1-knot group,
not much is known about how distinguish them.
For instance, it is known that the Alexander polynomials for 1-knots are very useful to distinguish 1-knots and they can be calculated easily.
On the other hand, the Alexander polynomials of almost all branched twist spins are equal to 1 and they cannot distinguish branched twist spins.
This is because the generators of the first elementary ideal of a branched twist spin interfere each other and the greatest common divisor in $\mathbb{Z}[t, t^{-1}]$ is equal to 1. 
Observing the first elementary ideals of branched twist spins directly, we can distinguish them sometimes~\cite{F1}.

In our previous study~\cite{F2}, we obtain an invariant of branched twist spins by using irreducible $\SLC$-metabelian representations.
This result is obtained as an analogue of counting conjugacy classes of irreducible $\SLC$-representations of 1-knots~\cite{L, N, NY}.
The facts that the fiber of $K^{m,n}$ is the punctured cyclic $m$-branched cover of $S^3$ along $K$ and the monodromy of the circle action is periodic of order $n$ are used in the proof.

In this paper, we are interested in representations of the knot group of a branched twist spin into finite groups.
A branched twist spin $K^{m,n}$ has a non-trivial center derived from a regular fiber if $m \neq 0,\pm1$.
The image of the center should be non-trivial, otherwise it becomes a representation of the knot group of the $1$-knot and it is not interesting.
Note that $K^{0,1}$ and $K^{\pm1,1}$ is a spun knot of $K$ and the trivial 2-knot, respectively.
For spun knots, including the trivial 2-knot, we can classify them by using the Gordon-Luecke theorem~\cite{GL}. Hence we study $K^{m,n}$ with $m \neq 0, \pm1$.
As representations with non-trivial center, we focus on $\SLZ$-representations and even-ordered dihedral group representations.
We will give a sufficient condition for the existence of an $\SLZ$-representation (Theorem~\ref{thm1}) and also determine the number of even-ordered dihedral group representations (Theorem~\ref{thmD2k}). 
As a consequence of the result about dihedral group representations, we have the following result.

\begin{thm}\label{Main}
Let $K_1$ and $K_2$ be non-trivial 1-knots and $K^{m_1,n_1}_1$ and $K^{m_2,n_2}_2$ be branched twist spins.
If $m_1$ and $m_2$ are different, then $K^{m_1,n_1}_1$ and $K^{m_2,n_2}_2$ are not equivalent.
\end{thm}

Note that two kinds of presentations of the knot group of a branched twist spin are known:
one is obtained from a Wirtinger presentation of the 1-knot and the other is obtained from the fiber surface of the branched twist spin.
The former presentation is used to show the results in this paper since the relations in the Wirtinger presentation are suitable to determine the representations. The latter presentation is used in our previous work~\cite{F2}. 

This paper is organized as follows. 
In Section 2, we will define a branched twist spin and give a certain property of $G$-representations of a branched twist spin with a group $G$ having a non-trivial center of degree 2.
In Section 3, we will check the order of each element in $\SLZ$ and observe the condition for the existence of  $\SLZ$-representations of branched twist spins concretely.
In Section 4, we will give even-ordered dihedral group representations of branched twist spins concretely and show Theorem~\ref{Main}. 

\noindent
{\bf Acknowledgement.} The author would like to express his appreciation to Masaharu Ishikawa for many valuable comments.

\section{branched twist spins}
To define branched twist spins, we first recall a classification of $S^1$-actions on $S^4$. See~\cite{Fi, Pa} for more details.
Consider an effective locally smooth $S^1$-action on $S^4$. 
It is known by Pao that any $S^1$-action is weak equivalent to the $S^1$-action whose orbit space is either the 3-ball $D^3$ or $S^3$.
If the orbit space is $S^3$, there exist at most 2 types of exceptional orbit.
Here the type of an orbit is determined by its isotropy group and an exceptional orbit is said to be $\mathbb{Z}_m$-type if its isotropy group is isomorphic to $\mathbb{Z}_m$.
Note that, if there exists two types of exceptional orbits, namely $\mathbb{Z}_m$-type and $\mathbb{Z}_n$-type, $m$ and $n$ are coprime.

Let $E_m$ and $E_n$ be the sets of all exceptional orbits whose types are  $\mathbb{Z}_m$ and $\mathbb{Z}_n$, respectively and let $F$ be the fixed point set.
Then the $(m,n)$-{\it branched twist spin}, denoted by $K^{m,n}$, is given as the union $E_n \cup F$.
Weak equivalent classes of effective locally smooth $S^1$-actions on $S^4$ are classified by Fintushel and Pao into four cases;
(1) $\{D^3\}$,
(2) $\{S^3\}$,
(3) $\{S^3, m\}$,
 and (4) $\{(S^3,K),m, n\}$.
Cases (1), (2) and (3) correspond to the trivial 2-knot, spun knots, $m$-twist spun knots, respectively.

The classification above is independent from the choice of the orientations of $S^4$ and orbits.
If we study branched twist spins with fixing the orientation of $S^4$,
we need to extend the numbers $m$ and $n$ to $(m,n)\in \mathbb{Z} \times \mathbb{N}$~\cite{F2}.

By using the $S^1$-action, we can give a presentation of the knot group of a branched twist spin, which will be explained now.
By definition,  the knot exterior of $\Kmn$ is the union of regular orbits and $\mathbb{Z}_m$-type orbits.
The union of regular orbits is homeomorphic to  $(S^3\setminus {\rm Int}N(K))\times S^1$, where $N(K)$ is a tubular neighborhood of the 1-knot $K$, and its fundamental group can be represented by a Wirtinger presentation of $K$ and the generator of $S^1$.
Let $\ve$ be the sign of $m$, i.e. $\ve = 1$ if $m\geq0$ and $\ve = -1$ if $m<0$. 
It is known in~\cite{F1,Pl2} that the knot group of $K^{m,n}$ has the following presentation:

\begin{equation}
\langle
 x_1, \ldots, ,x_l, h \mid r_1, \ldots, r_l, x_1hx_1^{-1}h^{-1}, \ldots, x_lhx_l^{-1}h^{-1}, x_1^{|m|}h^{\beta}
 \rangle\,,
 \end{equation}
where $\langle x_1, \ldots, x_l \mid r_1,\ldots, r_l\rangle$ is a Wirtinger presentation of $K$, $h$ is the generator corresponding to a free orbit and $\beta$ is an integer satisfying $n\beta \equiv \ve \ ({\rm mod}\ |m|)$.
Note that $\rho(x_i)^{|m|}h^{\beta} = 1$ holds, which can be derived from the relation $\rho(x_1)^{|m|}h^{\beta} = 1$ and the relations of the Wirtinger presentation.

The following lemma is useful for the latter sections.

\begin{lem}\label{keylemma}
Let $G$ be a group with non-trivial center $c$ whose order is 2.
Assume that a $G$-representation $\varphi : \pKmn \to G$ of $\Kmn$ satisfies $\varphi(h) = c$.
Then 

\begin{equation*}
(\varphi(h))^{\beta} =
\begin{cases}
1 & (m : {\rm odd}),\\
\varphi(h) & (m : {\rm even}).
\end{cases}
\end{equation*}
\end{lem}

\begin{proof}
It is suffice to show that we can choose $\beta$ so that $m$ and $\beta$ have different parities.
Since $\beta$ is an integer satisfying $n\beta \equiv \ve\  ({\rm mod}\ |m|)$, $\beta$ must be odd if $m$ is even.
If $m$ is odd and $\beta$ is also odd, let $\beta^{\prime} = \beta - m$.  
Then $\beta^{\prime}$ is even and $n\beta^{\prime} \equiv \ve\  ({\rm mod}\ |m|)$ still holds.
Moreover, the original presentation is equivalent to the presentation with replacing $\beta$ into $\beta^{\prime}$. Thus the assertion holds.
\end{proof}

\section{$\SLZ$-representations}

In this section, we will give a decomposition of $\SLZ$ by indices of its elements, and determine the number of $\SLZ$-representations of branched twist spins.
We write a matrix in $\SLZ$ as
$\begin{pmatrix}
a & b\\
c & d
\end{pmatrix}$ with $a,b,c,d \in \{-1, 0, 1\}$.

Put $\pSLZ = \SLZ \setminus \{\pm I\}$ and set
 \begin{eqnarray*}
P^+_3 =& \{A\in \pSLZ \ |\  A^3 = I\},\\
{P}^-_2 =& \{A\in \pSLZ \ |\  A^2 = -I\},\\
{P}^-_3 =& \{A\in \pSLZ \ |\  A^3 = -I\}.
\end{eqnarray*}
The elements of $\pSLZ$ are classified as follows:

\begin{lem}\label{decompSLZ}
\begin{equation*}
{P^-_2}=\left\{
\begin{pmatrix}
{0} & {1}\\
{-1} & {1}
\end{pmatrix},
\begin{pmatrix}
{0} & {-1}\\
{1} & {1}
\end{pmatrix},
\begin{pmatrix}
{1} & {1}\\
{1} & {-1}
\end{pmatrix},
\begin{pmatrix}
{-1} & {-1}\\
{-1} & {1}
\end{pmatrix},
\begin{pmatrix}
{-1} & {1}\\
{1} & {1}
\end{pmatrix},
\begin{pmatrix}
{1} & {-1}\\
{-1} & {-1}
\end{pmatrix}
\right\},
\end{equation*}

\begin{equation*}
P^+_3 = 
 \left\{
\begin{split}
&\begin{pmatrix}
{1} & {0}\\
{1} & {1}
\end{pmatrix},
\begin{pmatrix}
{1} & {0}\\
{-1} & {1}
\end{pmatrix},
\begin{pmatrix}
{1} & {1}\\
{0} & {1}
\end{pmatrix},
\begin{pmatrix}
{1} & {-1}\\
{0} & {1}
\end{pmatrix},\\
&\begin{pmatrix}
{0} & {1}\\
{-1} & {-1}
\end{pmatrix},
\begin{pmatrix}
{0} & {-1}\\
{1} & {-1}
\end{pmatrix},
\begin{pmatrix}
{-1} & {1}\\
{-1} & {0}
\end{pmatrix},
\begin{pmatrix}
{-1} & {-1}\\
{1} & {0}
\end{pmatrix}
\end{split}
\right\},
\end{equation*}

\begin{equation*}
P^-_3 =
\left\{
\begin{split}
&\begin{pmatrix}
{-1} & {0}\\
{1} & {-1}
\end{pmatrix},
\begin{pmatrix}
{-1} & {0}\\
{-1} & {-1}
\end{pmatrix},
\begin{pmatrix}
{-1} & {1}\\
{0} & {-1}
\end{pmatrix},
\begin{pmatrix}
{-1} & {-1}\\
{0} & {-1}
\end{pmatrix},\\
&\begin{pmatrix}
{0} & {1}\\
{-1} & {1}
\end{pmatrix},
\begin{pmatrix}
{0} & {-1}\\
{1} & {1}
\end{pmatrix},
\begin{pmatrix}
{1} & {-1}\\
{1} & {0}
\end{pmatrix},
\begin{pmatrix}
{1} & {1}\\
{-1} & {0}
\end{pmatrix}
\end{split}
\right\}.
\end{equation*}
Moreover, $\pSLZ = P^-_2 \sqcup P^+_3 \sqcup P^-_3$.

\end{lem}

\begin{proof}
Let $A = {\small \begin{pmatrix}
{a} & {b}\\
{c} & {d}
\end{pmatrix}}$
be a matrix in $\SLZ$.
By definition of $\SLZ$, ${\rm det}A \equiv 1$, i.e. ${a}{d}-{b}{c} \equiv {1}\ ({\rm mod}\ 3)$.
We divide the discussion into three cases: 
$({a}{d},{b}{c}) \equiv 
({1},{0}),({0},-{1})\  {\rm and}\ (-{1},{1})$.
If $({a}{d},{b}{c}) \equiv ({1},{0})$, 
then $({a},{d})$ is either $({1},{1})$ or $(-{1},-{1})$ and $({b},{c})$ is one of 
$({0},{0})$, 
$({0},{1})$,
 $({0},-{1})$, 
 $({1},{0})$ and
 $(-{1},{0})$.
In these cases, $A$ is one of 
{\small
\begin{equation}\label{bc0}
\begin{pmatrix}
{1} & {0}\\
{0} & {1}
\end{pmatrix},
\begin{pmatrix}
{1} & {0}\\
{1} & {1}
\end{pmatrix},
\begin{pmatrix}
{1} & {0}\\
-{1} & {1}
\end{pmatrix},
\begin{pmatrix}
{1} & {1}\\
{0} & {1}
\end{pmatrix},
\begin{pmatrix}
{1} & -{1}\\
{0} & {1}
\end{pmatrix},
\end{equation}

\begin{equation*}
\begin{pmatrix}
-{1} & {0}\\
{0} & -{1}
\end{pmatrix},
\begin{pmatrix}
-{1} & {0}\\
{1} & -{1}
\end{pmatrix},
\begin{pmatrix}
-{1} & {0}\\
-{1} & -{1}
\end{pmatrix},
\begin{pmatrix}
-{1} & {1}\\
{0} & -{1}
\end{pmatrix}  {\rm and}
\begin{pmatrix}
-{1} & -{1}\\
{0} & -{1}
\end{pmatrix}.
\end{equation*}
}

If  $({a}{d},{b}{c}) \equiv ({0},-{1})$,
then $({b},{c})$ is either $({1},-{1})$ or $(-{1},{1})$ and $({a},{d})$ is one of 
$({0},{0})$, 
$({0},{1})$,
$({0},-{1})$, 
$({1},{0})$ or 
$(-{1},{0})$.
In these cases, $A$ is one of 
{\small
\begin{equation}\label{ad0}
\begin{pmatrix}
{0} & {1}\\
-{1} & {0}
\end{pmatrix},\begin{pmatrix}
{0} & {1}\\
-{1} & {1}
\end{pmatrix},
\begin{pmatrix}
{0} & {1}\\
-{1} & -{1}
\end{pmatrix},
\begin{pmatrix}
{1} & {1}\\
-{1} & {0}
\end{pmatrix},
\begin{pmatrix}
-{1} & {1}\\
-{1} & {0}
\end{pmatrix},
\end{equation}

\begin{equation*}
\begin{pmatrix}
{0} & -{1}\\
{1} & {0}
\end{pmatrix},\begin{pmatrix}
{0} & -{1}\\
{1} & {1}
\end{pmatrix},
\begin{pmatrix}
{0} & -{1}\\
{1} & -{1}
\end{pmatrix},
\begin{pmatrix}
{1} & -{1}\\
{1} & {0}
\end{pmatrix} {\rm and}
\begin{pmatrix}
-{1} & -{1}\\
{1} & {0}
\end{pmatrix}.
\end{equation*}
}

If  $({a}{d},{b}{c}) \equiv (-{1},{1})$,
then $({a},{d})$ is either $({1},-{1})$ or $(-{1},{1})$ and $({b},{c})$ is either $({1},-{1})$ or $(-{1},{1})$.
In these cases, $A$ is one of 
{\small
\begin{equation}\label{tr0}
\begin{pmatrix}
{1} & {1}\\
{1} & -{1}
\end{pmatrix},\begin{pmatrix}
{1} & -{1}\\
{1} & -{1}
\end{pmatrix},
\begin{pmatrix}
-{1} & {1}\\
{1} & {1}
\end{pmatrix} {\rm and}
\begin{pmatrix}
-{1} & -{1}\\
-{1} & {1}
\end{pmatrix}.
\end{equation}}

\noindent
Thus $\SLZ$ consists of the elements in \eqref{bc0}, \eqref{ad0} and \eqref{tr0}.

Next we check the $i$-th powers of the matrices in \eqref{bc0}, \eqref{ad0} and \eqref{tr0}.
Assume that $A \in \pSLZ$ is one of \eqref{bc0}, i.e. the form of $A$ is  
{\small $
\begin{pmatrix}
\pm {1} & {b}\\
{c} & \pm {1}
\end{pmatrix},
$} where ${b}{c} \equiv 0$ and $({b},{c}) \not\equiv ({0},{0})$.
Then
{\small
\begin{equation*}
A^2 =
\begin{pmatrix}
\pm { 1} & {b}\\
{c} & \pm {1}
\end{pmatrix}
\begin{pmatrix}
\pm {1} & {b}\\
{c} & \pm {1}
\end{pmatrix} 
=
\begin{pmatrix}
{1} & \pm {2b}\\
\pm {2c} & {1}
\end{pmatrix}
\neq \pm I.
\end{equation*}

\begin{equation*}
A^3 =
\begin{pmatrix}
\pm {1} & {b}\\
{c} & \pm {1}
\end{pmatrix}
\begin{pmatrix}
{1} & \pm {2b}\\
\pm {2c} & {1}
\end{pmatrix} 
=
\begin{pmatrix}
\pm ({1+2bc}) & {3b}\\
{3c} & \pm ({1+2bc})
\end{pmatrix}
=
\begin{pmatrix}
\pm {1} & {0}\\
{0} & \pm {1}
\end{pmatrix}.
\end{equation*}
}

\noindent
Thus
\begin{equation*}
\begin{pmatrix}
 {1} & {0}\\
{1} &  {1}
\end{pmatrix},
\begin{pmatrix}
 {1} & {0}\\
-{1} &  {1}
\end{pmatrix},
\begin{pmatrix}
 {1} & {1}\\
{0} & {1}
\end{pmatrix},
\begin{pmatrix}
 {1} & -{1}\\
{0} & \ {1}
\end{pmatrix} \in P^+_3,
\end{equation*}

\begin{equation*}
\begin{pmatrix}
- {1} & {0}\\
{1} & - {1}
\end{pmatrix},
\begin{pmatrix}
- {1} & {0}\\
{-1} & - {1}
\end{pmatrix},
\begin{pmatrix}
- {1} & {1}\\
{0} & - {1}
\end{pmatrix},
\begin{pmatrix}
- {1} & {-1}\\
{0} & - {1}
\end{pmatrix} \in P^-_3.
\end{equation*}

Assume that $A \in \pSLZ$ is one of \eqref{ad0}, i.e. the form of $A$ is  
{\small $
\begin{pmatrix}
{a} & \pm {1}\\
\mp {1} & {d}
\end{pmatrix},
$} where ${a}{d} \equiv 0 $.
If ${a} \equiv {d} \equiv 0$, then 

{\small
\begin{equation*}
A^2 =
\begin{pmatrix}
{0} & \pm {1}\\
\mp {1} & {0}
\end{pmatrix}
\begin{pmatrix}
{0} & \pm {1}\\
\mp {1} & {0}
\end{pmatrix} 
=
\begin{pmatrix}
-{1} & {0}\\
 {0} & -{1}
\end{pmatrix}.
\end{equation*}
}

\noindent
Thus 
{\small
\begin{equation*}
\begin{pmatrix}
{0} &  {1}\\
- {1} & {0}
\end{pmatrix},
\begin{pmatrix}
{0} & -{1}\\
 {1} & {0}
\end{pmatrix}
 \in P^-_2 .
\end{equation*}
}

\noindent
If ${a} \equiv 0$ and ${d} \equiv \pm {1}$, we have
{\small
\begin{equation*}
A^2 =
\begin{pmatrix}
{0} & \pm {1}\\
\mp {1} & {d}
\end{pmatrix}
\begin{pmatrix}
{0} & \pm {1}\\
\mp {1} & {d}
\end{pmatrix} 
=
\begin{pmatrix}
-{1} & \pm {d}\\
\mp {d} & {-1+ d^2}
\end{pmatrix}
=
\begin{pmatrix}
-{1} & \pm {d}\\
\mp {d} & {0}
\end{pmatrix}
\neq \pm I,
\end{equation*}

\begin{equation*}
A^3 =
\begin{pmatrix}
{0} & \pm {1}\\
\mp {1} & {d}
\end{pmatrix}
\begin{pmatrix}
-{1} & \pm {d}\\
\mp {d} & {0}
\end{pmatrix} 
=
\begin{pmatrix}
-{d} & {0}\\
\mp ({ 1-d^2}) & -{d}
\end{pmatrix}
=
\begin{pmatrix}
-{d} & {0}\\
{0} & -{d}
\end{pmatrix} .
\end{equation*}
}

\noindent
If ${a} \equiv \pm {1}$ and ${d} \equiv {0}$, we have
{\small
\begin{equation*}
A^2 =
\begin{pmatrix}
{a} & \pm {1}\\
\mp {1} & {0}
\end{pmatrix}
\begin{pmatrix}
{a} & \pm {1}\\
\mp {1} & {0}
\end{pmatrix} 
=
\begin{pmatrix}
{a^2 - 1} & \pm {a}\\
\mp {a} & -{1}
\end{pmatrix}
=
\begin{pmatrix}
{0} & \pm {a}\\
\mp {a} & {0}
\end{pmatrix}
\neq \pm I,
\end{equation*}

\begin{equation*}
A^3 =
\begin{pmatrix}
{a} & \pm {1}\\
\mp {1} & {0}
\end{pmatrix}
\begin{pmatrix}
{0} & \pm {a}\\
\mp {a} & {-1}
\end{pmatrix} 
=
\begin{pmatrix}
-{a} & \pm ({a^2 -1})\\
 {0} & -{a}
\end{pmatrix}
=
\begin{pmatrix}
-{a} & {0}\\
{0} & -{a}
\end{pmatrix} .
\end{equation*}
}

\noindent
Thus 
{\small 
\begin{equation*}
\begin{pmatrix}
{0} & {1}\\
-{1} & -{1}
\end{pmatrix},
\begin{pmatrix}
{0} & -{ 1}\\
{1} & -{1}
\end{pmatrix},
\begin{pmatrix}
-{1} & { 1}\\
-{1} & {0}
\end{pmatrix},
\begin{pmatrix}
-{1} & -{ 1}\\
{1} & {0}
\end{pmatrix} \in P^+_3,
\end{equation*}
\begin{equation*}
\begin{pmatrix}
{0} & {1}\\
-{1} & {1}
\end{pmatrix},
\begin{pmatrix}
{0} & -{1}\\
{1} & {1}
\end{pmatrix},
\begin{pmatrix}
{1} & {1}\\
-{1} & {0}
\end{pmatrix},
\begin{pmatrix}
{1} & -{1}\\
{1} & {0}
\end{pmatrix} \in P^-_3.
\end{equation*}
}

Finally, assume the form of $A$ is one of \eqref{tr0}. Then tr$(A) \equiv 0$, ${a^2} \equiv {d^2} \equiv 1$ and ${bc} \equiv 1$. So we obtain     
\begin{equation*}
A^2 =
{\small
\begin{pmatrix}
{a} & {b}\\
{c} & {d}
\end{pmatrix}
\begin{pmatrix}
{a} & {b}\\
{c} & {d}
\end{pmatrix} 
=
\begin{pmatrix}
{a^2+ bc} & {b(a+d)}\\
{c(a+d)} & {bc+d^2}
\end{pmatrix}
=
\begin{pmatrix}
-{1} & {0}\\
{0} & -{1}
\end{pmatrix}
}.
\end{equation*}
Thus
{\small
\begin{equation}\label{A2ber}
\begin{pmatrix}
{1} & {1}\\
{1} & -{1}
\end{pmatrix},\begin{pmatrix}
{1} & -{1}\\
-{1} & -{1}
\end{pmatrix},
\begin{pmatrix}
-{1} & {1}\\
{1} & {1}
\end{pmatrix},
\begin{pmatrix}
-{1} & -{1}\\
-{1} & {1}
\end{pmatrix}\in P^-_2,
\end{equation}}%
and this completes the proof.
\end{proof}

\begin{lem}\label{quasiconj}
For $A,B \in P^-_2$,
\begin{equation*}
ABA =
\begin{cases}
B^{-1} & (B = A^{\pm 1})\\
B & (B \neq A^{\pm 1}).
\end{cases} 
\end{equation*}
\end{lem}

\begin{proof}
From Lemma~\ref{decompSLZ}, $P^-_2$ consists of 
{\small \begin{equation*}
\begin{pmatrix}
{0} &  {1}\\
- {1} & {0}
\end{pmatrix},
\begin{pmatrix}
{0} & -{1}\\
 {1} & {0}
\end{pmatrix},
\begin{pmatrix}
{1} & {1}\\
{1} & -{1}
\end{pmatrix},\begin{pmatrix}
{1} & -{1}\\
-{1} & -{1}
\end{pmatrix},
\begin{pmatrix}
-{1} & {1}\\
{1} & {1}
\end{pmatrix} {\rm and}
\begin{pmatrix}
-{1} & -{1}\\
-{1} & {1}
\end{pmatrix},
\end{equation*}
}
and we can easily check that 
{\small 
\begin{equation*}
\begin{pmatrix}
{0} & {1}\\
-{1} & {0}
\end{pmatrix}
=
\begin{pmatrix}
{0} & -{1}\\
{1} & {0}
\end{pmatrix}^{-1},
\begin{pmatrix}
{1} & {1}\\
{1} & -{1}
\end{pmatrix}
=
\begin{pmatrix}
-{1} & -{1}\\
-{1} & {1}
\end{pmatrix}^{-1} {\rm and}
\begin{pmatrix}
{1} &-{1}\\
-{1} & -{1}
\end{pmatrix}
=
\begin{pmatrix}
-{1} & {1}\\
{1} & {1}
\end{pmatrix}^{-1}.
\end{equation*} 
}
If 
$A =  {\small
\begin{pmatrix}
{0} & {1}\\
-{1} & {0}
\end{pmatrix}},
$ then

{\small
\begin{equation*}
\begin{split}
\begin{pmatrix}
{0} & {1}\\
-{1} & {0}
\end{pmatrix}
\begin{pmatrix}
{0} & {1}\\
-{1} & {0}
\end{pmatrix}
\begin{pmatrix}
{0} & {1}\\
-{1} & {0}
\end{pmatrix} 
=&
\begin{pmatrix}
-{1} & {0}\\
{0} & -{1}
\end{pmatrix}
\begin{pmatrix}
{0} & {1}\\
-{1} & {0}
\end{pmatrix}
=
\begin{pmatrix}
{0} & -{1}\\
{1} & {0}
\end{pmatrix},\\
\begin{pmatrix}
{0} & {1}\\
-{1} & {0}
\end{pmatrix}
\begin{pmatrix}
{0} & -{1}\\
{1} & {0}
\end{pmatrix}
\begin{pmatrix}
{0} & {1}\\
-{1} & {0}
\end{pmatrix} 
=&
\begin{pmatrix}
{1} & {0}\\
{0} & {1}
\end{pmatrix}
\begin{pmatrix}
{0} & {1}\\
-{1} & {0}
\end{pmatrix}
=
\begin{pmatrix}
{0} & {1}\\
-{1} & {0}
\end{pmatrix},\\
\begin{pmatrix}
{0} & {1}\\
-{1} & {0}
\end{pmatrix}
\begin{pmatrix}
{1} & {1}\\
{1} & -{1}
\end{pmatrix}
\begin{pmatrix}
{0} & {1}\\
-{1} & {0}
\end{pmatrix} 
=&
\begin{pmatrix}
{1} & -{1}\\
-{1} & -{1}
\end{pmatrix}
\begin{pmatrix}
{0} & {1}\\
-{1} & {0}
\end{pmatrix}
=
\begin{pmatrix}
{1} & {1}\\
{1} & -{1}
\end{pmatrix},\\
\begin{pmatrix}
{0} & {1}\\
-{1} & {0}
\end{pmatrix}
\begin{pmatrix}
-{1} & -{1}\\
-{1} & {1}
\end{pmatrix}
\begin{pmatrix}
{0} & {1}\\
-{1} & {0}
\end{pmatrix} 
=&
\begin{pmatrix}
-{1} & {1}\\
{1} & {1}
\end{pmatrix}
\begin{pmatrix}
{0} & {1}\\
-{1} & {0}
\end{pmatrix}
=
\begin{pmatrix}
-{1} & -{1}\\
-{1} & {1}
\end{pmatrix},\\
\begin{pmatrix}
{0} & {1}\\
-{1} & {0}
\end{pmatrix}
\begin{pmatrix}
{1} & -{1}\\
-{1} & -{1}
\end{pmatrix}
\begin{pmatrix}
{0} & {1}\\
-{1} & {0}
\end{pmatrix} 
=&
\begin{pmatrix}
{1} & -{1}\\
-{1} & -{1}
\end{pmatrix}
\begin{pmatrix}
{0} & {1}\\
-{1} & {0}
\end{pmatrix}
=
\begin{pmatrix}
{1} & {1}\\
{1} & -{1}
\end{pmatrix},\\
\begin{pmatrix}
{0} & {1}\\
-{1} & {0}
\end{pmatrix}
\begin{pmatrix}
-{1} & {1}\\
{1} & {1}
\end{pmatrix}
\begin{pmatrix}
{0} & {1}\\
-{1} & {0}
\end{pmatrix} 
=&
\begin{pmatrix}
{1} & {1}\\
{1} & -{1}
\end{pmatrix}
\begin{pmatrix}
{0} & {1}\\
-{1} & {0}
\end{pmatrix}
=
\begin{pmatrix}
-{1} & {1}\\
{1} & {1}
\end{pmatrix}.
\end{split}
\end{equation*}
}

For the other cases 
$A = 
\begin{pmatrix}
{0} & -{1}\\
{1} & {0}
\end{pmatrix},
\begin{pmatrix}
{1} & {1}\\
{1} & -{1}
\end{pmatrix}.
\begin{pmatrix}
-{1} & -{1}\\
-{1} & {1}
\end{pmatrix},
\begin{pmatrix}
{1} &-{1}\\
-{1} & -{1}
\end{pmatrix} {\rm and}
\begin{pmatrix}
-{1} & {1}\\
{1} & {1}
\end{pmatrix}, 
$
we can also calculate $ABA$ as in Table~\ref{iji}. 
Thus the assertion holds.
\end{proof}
\ \\
 \begin{equation*}
\begin{array}[h]{|c||c|c|c|c|c|c|}  \hline
   A\backslash B  &\begin{pmatrix}
{0} & {1}\\
{-1} & {0}
\end{pmatrix}&
\begin{pmatrix}
{0} & {-1}\\
{1} & {0}
\end{pmatrix}&
\begin{pmatrix}
{1} & {1}\\
{1} & {-1}
\end{pmatrix}&
\begin{pmatrix}
{-1} & {-1}\\
{-1} & {1}
\end{pmatrix}&
\begin{pmatrix}
{1} & {-1}\\
{-1} & {-1}
\end{pmatrix}&
\begin{pmatrix}
{-1} & {1}\\
{1} & {1}
\end{pmatrix} \\ \hline \hline

\begin{pmatrix}
{0} & {1}\\
{-1} & {0}
\end{pmatrix}&
\begin{pmatrix}
{0} & {-1}\\
{1} & {0}
\end{pmatrix}&
\begin{pmatrix}
{0} & {1}\\
{-1} & {0}
\end{pmatrix}&
\begin{pmatrix}
{1} & {1}\\
{1} & {-1}
\end{pmatrix}&
\begin{pmatrix}
{-1} & {-1}\\
{-1} & {1}
\end{pmatrix}&
\begin{pmatrix}
{1} & {-1}\\
{-1} & {-1}
\end{pmatrix}&
\begin{pmatrix}
{-1} & {1}\\
{1} & {1}
\end{pmatrix} \\ \hline 

\begin{pmatrix}
{0} & {-1}\\
{1} & {0}
\end{pmatrix}&
\begin{pmatrix}
{0} & {-1}\\
{1} & {0}
\end{pmatrix}&
\begin{pmatrix}
{0} & {1}\\
{-1} & {0}
\end{pmatrix}&
\begin{pmatrix}
{1} & {1}\\
{1} & {-1}
\end{pmatrix}&
\begin{pmatrix}
{-1} & {-1}\\
{-1} & {1}
\end{pmatrix}&
\begin{pmatrix}
{1} & {-1}\\
{-1} & {-1}
\end{pmatrix}&
\begin{pmatrix}
{-1} & {1}\\
{1} & {1}
\end{pmatrix} \\ \hline 

\begin{pmatrix}
{1} & {1}\\
{1} & {-1}
\end{pmatrix}&
\begin{pmatrix}
{0} & {1}\\
{-1} & {0}
\end{pmatrix}&
\begin{pmatrix}
{0} & {-1}\\
{1} & {0}
\end{pmatrix}&
\begin{pmatrix}
{-1} & {-1}\\
{-1} & {1}
\end{pmatrix}&
\begin{pmatrix}
{1} & {1}\\
{1} & {-1}
\end{pmatrix}&
\begin{pmatrix}
{1} & {-1}\\
{-1} & {-1}
\end{pmatrix}&
\begin{pmatrix}
{-1} & {1}\\
{1} & {1}
\end{pmatrix} \\ \hline 

\begin{pmatrix}
{-1} & {-1}\\
{-1} & {1}
\end{pmatrix}&
\begin{pmatrix}
{0} & {1}\\
{-1} & {0}
\end{pmatrix}&
\begin{pmatrix}
{0} & {-1}\\
{1} & {0}
\end{pmatrix}&
\begin{pmatrix}
{-1} & {-1}\\
{-1} & {1}
\end{pmatrix}&
\begin{pmatrix}
{1} & {1}\\
{1} & {-1}
\end{pmatrix}&
\begin{pmatrix}
{1} & {-1}\\
{-1} & {-1}
\end{pmatrix}&
\begin{pmatrix}
{-1} & {1}\\
{1} & {1}
\end{pmatrix} \\ \hline 

\begin{pmatrix}
{1} & {-1}\\
{-1} & {-1}
\end{pmatrix}&
\begin{pmatrix}
{0} & {1}\\
{-1} & {0}
\end{pmatrix}&
\begin{pmatrix}
{0} & {-1}\\
{1} & {0}
\end{pmatrix}&
\begin{pmatrix}
{1} & {1}\\
{1} & {-1}
\end{pmatrix}&
\begin{pmatrix}
{-1} & {-1}\\
{-1} & {1}
\end{pmatrix}&
\begin{pmatrix}
{-1} & {1}\\
{1} & {1}
\end{pmatrix}&
\begin{pmatrix}
{1} & {-1}\\
{-1} & {-1}
\end{pmatrix} \\ \hline 

\begin{pmatrix}
{-1} & {1}\\
{1} & {1}
\end{pmatrix}&
\begin{pmatrix}
{0} & {1}\\
{-1} & {0}
\end{pmatrix}&
\begin{pmatrix}
{0} & {-1}\\
{1} & {0}
\end{pmatrix}&
\begin{pmatrix}
{1} & {1}\\
{1} & {-1}
\end{pmatrix}&
\begin{pmatrix}
{-1} & {-1}\\
{-1} & {1}
\end{pmatrix}&
\begin{pmatrix}
{-1} & {1}\\
{1} & {1}
\end{pmatrix}&
\begin{pmatrix}
{1} & {-1}\\
{1} & {-1}
\end{pmatrix} \\ \hline 
\end{array}
\end{equation*}
\begin{table}[h]
\caption{Calculations of $ABA$}
\label{iji}
\end{table}

Now we will study $\SLZ$-representations of branched twist spins.
It is known that the group $\SLZ$ has the following presentation:
\begin{equation*}
\langle a,b,c \ | \ a^3 = b^3= c^2 = abc \rangle,
\end{equation*}
where the element $abc$ is the non-trivial center in $\SLZ$. The order of $abc$ is equal to 2 since $\SLZ$ is isomorphic to the group whose presentation is given as
\begin{equation*}
\langle a_1,b_1,c_1 \ | \ a_1^2 = b_1^2= c_1^2 = a_1b_1c_1 \rangle,
\end{equation*}
 where $a_1 = aca^{-1}$, $b_1= c$ and $c_1 = c^{-1}bcb^{-1}c$, and this is a presentation of the dicyclic group with parameter $2$.
Note that the order of $\SLZ$ is $24$, and $-I$ is the only non-trivial center.

The number of  representations $\rho: \pKmn \rightarrow \SLZ$ satisfying that $\rho(abc) = -I$, where $I$ is the identity matrix in $\SLZ$, is determined as in the following theorem.

\begin{thm}\label{thm1}
Let $K^{m,n}$ be a branched twist spin and $\rho: \pKmn \to \SLZ$ be an $\SLZ$-representation of $\Kmn$ with $\rho(h) = -I$.
\begin{itemize}
\item[(1)] If $m$ is odd and $3  \mid \hspace{-8pt}/\  m$, then $\rho(x_i) = I$ for all $i$, i.e. $\#\{\rho\} = 1$.
\item[(2)] If $4 \mid m$, such $\rho$ does not exist, i.e. $\#\{\rho\} = 0$.
\item[(3)]  If $2 \mid \hspace{-8pt}/\ (m/2)$, then $\rho(x_i) = \rho(x_j)$ and $(\rho(x_i))^2= -I$ for all $i,j$,  i.e. $\#\{\rho\} = |{A}_2| = 6$.
\end{itemize}
\end{thm}

\begin{proof}[Proof of Theorem~\ref{thm1}]
To prove (1), it is suffice to show that $\rho(x_i) = I$ for $i = 1,\ldots,l$.
Since $m$ is odd, $(\rho(x_i))^{|m|} = I$ follows from $(\rho(x_i))^{|m|}(\rho(h))^{\beta}=1$ and Lemma~\ref{keylemma}.
Thus $\rho(x_i) = I$ or $\rho(x_i) \in P^+_3 \sqcup P^-_3$ for any $i$ by Lemma~\ref{decompSLZ}.
If $\rho(x_j) \neq I$ for some $j$, then $\rho(x_j) \in P^+_3 \sqcup P^-_3$.
Since $3  \mid \hspace{-9pt}/\  m$, $\rho(x_j)$ does not satisfy $\rho(x_j)^{|m|}= I$ by Lemma~\ref{decompSLZ}.
Therefore  $\rho(x_i) = I$ for $i=1,\ldots, l$.

To prove (2) and (3), we assume that $m$ is even.
From Lemma~\ref{keylemma} and $(\rho(x_i))^{|m|}(\rho(h))^{\beta}=1$, $(\rho(x_i))^{|m|} = -I$ holds for $i= 1,\ldots,l$,
that is, $\rho(x_i) \in P^-_2$.

If $ 4 \mid m$, then $(\rho(x_i))^{|m|}= I$ holds, and it is a contradiction and (2) holds.

If $2  \mid \hspace{-9pt}/\ (m/2)$, for a relator $x_a x_b x^{-1}_a x^{-1}_c$ 
of the Wirtinger presentation, by Lemma~\ref{quasiconj}, the following equiality holds

\begin{eqnarray}\label{wirrel}
&1 =&\left(\rho(x_a)\rho(x_b)(\rho(x_a)^{-1})\right) \left(\rho(x_a)\rho(x_b)\rho(x_a)\right) \\
&\ \  =& 
\begin{cases}\nonumber
\rho(x_c)(\rho(x_b))^{-1} & ( \rho(x_a) = (\rho(x_b))^{\pm 1})   \\
\rho(x_c)\rho(x_b) & ( \rho(x_a) \neq (\rho(x_b))^{\pm 1}).
\end{cases}
\end{eqnarray}
Applying the same observation to all relators of the Wirtinger presentation, we obtain
 $\rho(x_i) =  \rho(x_j)$ or $\rho(x_i) =  (\rho(x_j))^{-1}$ for any $i,j$.
Hence, by~\eqref{wirrel},  $\rho(x_i) =  \rho(x_j)$ for any $i,j$.
Thus the number of $\{\rho\}$ is equal to $|P^-_2|$, and it is 6 by Lemma~\ref{decompSLZ}.  
\end{proof}

\section{Dihedral group representattions}
Let $D_{2k}$ be a dihedral group of order $2k$ ($k\in \mathbb{N}$).
It is known that $D_{2k}$ has the following presentation:
\begin{equation*}
\langle r,s \ | \ r^{2k},  s^2, rsrs \rangle.
\end{equation*}
Since the order of any element is even, $r^k$ is the non-trivial center with order $2$.
Let $\tau : \pKmn \to D_{2k}$ be a representation such that $\tau(h) = r^k$. 
From the relators $x^{|m|}_i h^{\beta}$ for $i = 1,\ldots, l$ and Lemma~\ref{keylemma}, we obtain
\begin{equation}\label{relDi}
1 = (\tau(x_i))^{|m|}(\tau(h))^{\beta}=
\begin{cases}
 (\tau(x_i))^{|m|}r^k & (\forall i\ {\rm if }\ m: {\rm even})\\
 (\tau(x_i))^{|m|} & (\forall i\ {\rm if }\ m: {\rm odd}).
\end{cases}
\end{equation}

Set $t^{(j)}_i \in \{r, r^{-1}, s\}$. 
Put $\prod_j t^{(i)}_j$ to be a word of $\tau(x_i)$.
The relators $rsrs$ and $s^2$ imply $rsr=s$. Using this, the word $\prod_j t^{(i)}_j$ can be changed into the form $r^{p_i}s^{\delta}$, where $0\leq p_i<2k$ is an integer  and $\delta=1$ if the sum of indices of $s$ in $\prod_j t^{(i)}_j$ is odd and $\delta=0$ if that is even. 
If $\tau(x_i) =  r^{p_i}s$ for some $i$, using $rsr=s$, we have
\begin{equation*}
(\tau(x_i))^{|m|} =
\begin{cases}
1& \ {\rm if }\ m: {\rm even}\\
r^{p_i}s& \ {\rm if }\ m: {\rm odd}.
\end{cases}
\end{equation*}
Therefore \eqref{relDi} can be rewritten as 
\begin{equation}\label{rs}
1 = (\tau(x_i))^{|m|}(\tau(h))^{\beta} =
\begin{cases}
r^k& \ {\rm if }\ m: {\rm even}\\
r^{p_i} s& \ {\rm if }\ m: {\rm odd},
\end{cases}
\end{equation}
However $r^{k},s \neq 1$ and $s$ cannot be written as a product of $r$.
This is a contradiction, and $\tau(x_i)$ must be in the form $r^{p_i}$.
From a relator $r_i = x_a x_b x^{-1}_ax^{-1}_c$ with $\rho(x_b)= r^{p_b}$ and $\rho(x_c)= r^{p_c}$, we obtain
\begin{equation*}
1 = \tau(r_i)= \tau(x_a)\tau(x_b)(\tau(x_a))^{-1}(\tau(x_c))^{-1} = r^{p_b-p_c}.
\end{equation*} 
The same conclusion is obtained for all $x_i$, $i= 1,\ldots, l$ and hence $\rho(x_i) = r^p$, where $p$ is an integer independent of $i$. 
Thus \eqref{relDi} can be rewritten as 
\begin{equation}\label{realrel}
1 = (\tau(x_i))^{|m|}(\tau(h))^{\beta} =
\begin{cases}
r^{|m|p+k}& \ {\rm if }\ m: {\rm even}\\
r^{|m|p}& \ {\rm if }\ m: {\rm odd},
\end{cases}
\end{equation}

 By using \eqref{realrel}, we have the following theorem.

\begin{thm}\label{thmD2k}
Let $K^{m,n}$ be a branched twist spin and $\tau :\pKmn \to D_{2k}$ be a $D_{2k}$-presentation of $\Kmn$ with $\tau(h) = r^k$.
Put $d = {\rm gcd}(|m|,k)$.\\
(1) If $m$ is even and  $d$ is odd, then such $\tau$ does not exist, i.e. $\#\{\tau\}=0$.\\
(2) If $m$ is even and $d$ is also even, then $\#\{\tau\} = d/2$.\\
(3) If $m$ is odd, then $\#\{\tau\} = (d+1)/2$.
\end{thm}

\begin{proof}[Proof of Theorem~\ref{thmD2k}]
We first assume that $m$ is even.
In this case, from \eqref{realrel}, $|m|p \equiv k\ ({\rm mod}\ 2k)$ holds and hence
\begin{equation}\label{coevenrel}
|m|p = (2t+1)k
\end{equation}
holds for some $t\in \mathbb{Z}$.
It is necessary that $k$ is even since $m$ is even. 
Since $d = {\rm gcd}(|m|,k)$, $d$ is never odd and (1) holds. 
Now we consider the case that $d$ is even.
Let $m^{\prime}$ and $k^{\prime}$ be integers such that $m = m^{\prime}d$ and $k = k^{\prime}d$.
Then \eqref{coevenrel} is rewritten as
\begin{equation}\label{evenrel}
|m^{\prime}|p = (2t+1)k^{\prime},
\end{equation}
Note that $0\leq p<2k$.
If $m^{\prime}$ is even, then $k^{\prime}$ is also even, and it contradicts for $d = {\rm gcd}(|m|,k)$.
Thus we can assume that $m^{\prime}$ is odd.
In this case we can choose $p$ as $k^{\prime},3k^{\prime},\ldots, (d-1)k^{\prime}$, and the number of choices of $p$ is $d/2$.
This completes the proof of (2).  

Next, we assume that $m$ is odd.
From \eqref{realrel},  $|m|p \equiv 0\ ({\rm mod}\ 2k)$ holds.
Let $m^{\prime}$ and $k^{\prime}$ be integers such that $m = m^{\prime}d$ and $k = k^{\prime}d$.
Then $|m|p \equiv 0\ ({\rm mod}\ 2k)$ is rewritten as 
\begin{equation}\label{oddrel}
|m^{\prime}|p \equiv 0\ ({\rm mod}\ 2k^{\prime}).
\end{equation}
Note that $0\leq p<2k$ and $d$ is odd. 
	Since $m^{\prime}$ and $k^{\prime}$ are coprime, we obtain $p \equiv 0\ ({\rm mod}\ 2k^{\prime})$ from \eqref{oddrel}.
Then we can choose $p$ as $1,2k^{\prime}, \ldots (d-1)k^{\prime}$ and the number of choises of $p$ is $(d+1)/2$.
\end{proof}

\begin{proof}[Proof of Theorem \ref{Main}]
Before comparing representations of $K_1^{m_1, n_1}$ and $K_2^{m_2, n_2}$, we first observe spun knots and 1-twist spun knots of any 1-knot $K$.
In case of $m = 0$, $h$ is the unit in $\pKmn$ and $\pKmn$ is isomorphic to the knot group of $K$. 
Thus $\tau(h)$ is never $r^k$ for any $k$, and hence $\#\{\tau\} = 0$.  
If we define ${\rm gcd}(0,k)=0$, the conclusion (2) of Theorem~\ref{thmD2k} holds.
In case of $m = \pm 1$, $\pKmn$ is an infinite cyclic group and $h$ is a generator of $\pKmn$.
Thus $\tau(h) = r^k$, and $\#\{\tau\} = 1$. Hence the conclusion (3) of Theorem~\ref{thmD2k} holds.
Note that $\pKmn$ is an infinite cyclic group for any $m,n$ if $K$ is a trivial 1-knot.
This is why we assume that $K_1$ and $K_2$ are non-trivial 1-knots.

From now, we set $d = {\rm gcd}(0,k)= 0$ if $m=0$.  
Then the conclusion in Theorem~\ref{thmD2k} holds for any $\Kmn$ with $(m,n)\in \mathbb{Z}\times \mathbb{N}$.
Let 
\begin{eqnarray*}
\tau_1 :\pi_1(S^4\setminus {\rm Int}(K_1^{m_1,n_1})) \to D_{2|m_1|}& {\rm and} \\
\tau_2 :\pi_1(S^4\setminus {\rm Int}(K_2^{m_2,n_2})) \to D_{2|m_1|}& 
\end{eqnarray*}
 be $D_{2|m_1|}$-representations of $K_1^{m_1,n_1}$ and $K_2^{m_2, n_2}$, respectively, 
 where $m_1 \neq m_2$, and set $d^{\prime} = {\rm gcd}(m_1,m_2)$. 
 
First we assume $m_1$ and $m_2$ have the same parity and $|m_1|> |m_2|$. 
Then, from Theorem~\ref{thmD2k}, the number of $\tau_1$ is equal to $|m_1|/2$ or $(|m_1| + 1)/2$ while the number of $\tau_2$ is equal to $d^{\prime}/2$ or $(d^{\prime} +1)/2$.
Thus $\#\{\tau_1\} \neq \#\{\tau_2\}$ and hence $K_1^{m_1,n_1}$ and $K_2^{m_2,n_2}$ are not equivalent.

Next, we assume that $m_1$ and $m_2$ have different parities and $m_1$ is odd. 
In this case, from Theorem~\ref{thmD2k} again,  $\#\{\tau_1\} = (|m_1|+1)/2$ while $\#\{\tau_2\} =0$ since $d^{\prime}$ is odd.
This completes the proof.
\end{proof}

\end{document}